\documentclass[12pt, reqno]{amsart}
\usepackage{amsmath, amsthm, amscd, amsfonts, amssymb, graphicx, color}
\usepackage[bookmarksnumbered, colorlinks, plainpages]{hyperref}

\newtheorem{theorem}{Theorem}[section]
\newtheorem{lemma}[theorem]{Lemma}

\theoremstyle{definition}
\newtheorem{definition}[theorem]{Definition}

\theoremstyle{remark}

\numberwithin{equation}{section}
\begin{document}
\title[A Note on Inextensible Flows of Curves in $E_{{}}^{n}$]{A Note on Inextensible Flows of Curves in $E_{{}}^{n}$}
\author[\"{O}nder G\"{o}kmen Y\i ld\i z, Murat Tosun, S\i dd\i ka \"{O}zkald\i \ Karaku\c{s}]{\"{O}nder G\"{o}kmen Y\i ld\i z, Murat Tosun, S\i dd\i ka \"{O}zkald\i \ Karaku\c{s}}

\address{$^{1}$  Bilecik University, Faculty of Sciences and Arts, Department of Mathematics, 11210, Bilecik, Turkey.}

\email{ogokmen.yildiz@bilecik.edu.tr}
\email{siddika.karakus@bilecik.edu.tr}

\address{$^{2}$ Sakarya, Faculty of Sciences and Arts, Department of Mathematics, Sakarya, Turkey}
\email{tosun@sakarya.edu.tr}


\begin{abstract}
In this paper, we investigate the general formulation for inextensible flows
of curves in $E_{{}}^{n}$. The necessary and sufficient conditions for
inextensible curve flow are expressed as a partial differential equation
involving the curvatures.

\end{abstract}
\maketitle

\section{Introduction}

It is well known that many nonlinear phenomena in physics, chemistry and
biology are described by dynamics of shapes, such as curves and surfaces. The
evolution of curve and surface has significant applications in computer vision
and image processing. The time evolution of a curve or surface generated by
its corresponding flow in -for this reason we shall also refer to curve and
surface evolutions as flows throughout this article- is said to be
inextensible if, in the former case, its arclength is preserved, and in the
latter case, if its intrinsic curvature is preserved \cite{kwon}. Physically,
inextensible curve flows give rise to motions in which no strain energy is
induced. The swinging motion of a cord of fixed length, for example, or of a
piece of paper carried by the wind, can be described by inextensible curve and
surface flows. Such motions arise quite naturally in a wide range of a
physical applications. For example, both Chirikjian and Burdick
\cite{chirikjian} and Mochiyama et al. \cite{mochiyama} study the shape
control of hyper-redundant, or snake-like robots.

Inextinsible curve and surface flows also arise in the context of many
problems in computer vision \cite{kass}\cite{lu} and computer animation
\cite{desbrun}, and even structural mechanics \cite{unger}.

There have been a lot of studies in the literature on plane curve flows,
particularly on evolving curves in the direction of their curvature vector
field (referred to by various names such as \textquotedblleft curve
shortening\textquotedblright, flow by curvature\textquotedblright \ and
\textquotedblleft heat flow\textquotedblright). Particularly relevant to this
paper are the methods developed by Gage and Hamilton \cite{gage} and Grayson
\cite{grayson} for studying the shrinking of closed plane curves to circle via
heat equation.

In this paper, we develop the general formulation for inextensible flows of
curves in $E^{n\text{ }}$. Necessary and sufficient conditions for an
inextensible curve flow are expressed as a partial differential equation
involving the curvatures.

\section{Preliminary}

To meet the requirements in the next sections, the basic elements of the
theory of curves in the Euclidean n-space $E^{n}$ are briefly presented in
this section (A more complete elementary treatment can be found in
\cite{glunk}\cite{hhh}).

Let $\alpha:I\subset R\mathbb{\longrightarrow}E^{n\text{ }}$be an arbitrary
curve in $E^{n\text{ }}$ Recall that the curve $\alpha$ is said to be a unit
speed curve (or parameterized by arclength functions) if $\left \langle
\alpha^{\prime}(s),\alpha^{\prime}(s)\right \rangle =1$, where $\left \langle
.,.\right \rangle $\ denotes the standard inner product of given by%
\[
\left \langle X,Y\right \rangle =%
{\displaystyle \sum \limits_{i=1}^{n}}
x_{i}y_{i}%
\]

for each $X=\left(  x_{1},x_{2},...,x_{n}\right)  ,Y=\left(  y_{1}%
,y_{2},...,y_{n}\right)  \in R^{n}$. In particular, norm of a vector $X\in
R^{n}$ is given by $\left \Vert X\right \Vert =\sqrt{\left \langle
X,Y\right \rangle }$. Let $\left \{  V_{1},V_{2},...,V_{n}\right \}  $\ be the
moving Frenet frame along the unit speed curve $\alpha$, where $V_{i}\left(
i=1,2,...,n\right)  $\ denotes the $i^{th}$ Frenet vector field. Then Frenet
formulas are given by%
\[
\left[
\begin{array}
[c]{c}%
V_{1}^{\prime}\\
V_{2}^{\prime}\\
V_{3}^{\prime}\\
\vdots \\
V_{n-2}^{\prime}\\
V_{n-1}^{\prime}\\
V_{n}^{\prime}%
\end{array}
\right]  =\left[
\begin{array}
[c]{cccccccc}%
0 & k_{1} & 0 & 0 & \cdots & 0 & 0 & 0\\
-k_{1} & 0 & k_{2} & 0 & \cdots & 0 & 0 & 0\\
0 & -k_{2} & 0 & k_{3} & \cdots & 0 & 0 & 0\\
\vdots & \vdots & \vdots & \vdots & \vdots & \vdots & \vdots & \vdots \\
0 & 0 & 0 & 0 & \cdots & 0 & k_{n-2} & 0\\
0 & 0 & 0 & 0 & \cdots & -k_{n-2} & 0 & k_{n-1}\\
0 & 0 & 0 & 0 & \cdots & 0 & -k_{n-1} & 0
\end{array}
\right]  \left[
\begin{array}
[c]{c}%
V_{1}\\
V_{2}\\
V_{3}\\
\vdots \\
V_{n-2}\\
V_{n-1}\\
V_{n}%
\end{array}
\right]
\]

where $k_{i}\left(  i=1,2,...,n\right)  $ denotes the $i^{th}$ curvature
function of the curve \cite{glunk}\cite{hhh}. If all of the curvatures
$k_{i}\left(  i=1,2,...,n\right)  $ of the curve vanish nowhere in $I\subset
R$, it is called a non-degenerate curve.

\section{Inextinsible Flows of Curve in $E^{n}$}

Throughout this paper, we suppose that
\[
\alpha:\left[  0,l\right]  \times \left[  0,w\right)  \mathbb{\longrightarrow
}E^{n\text{ }}%
\]
is a one parameter family of smooth curves in $E^{n\text{ }}$, where $l$\ is
the arclength of the initial curve. Let $u$\ be the curve parameterization
variable, $0\leq u\leq l$. If the speed curve $\alpha$\ is denoted by
$v=\left \Vert \frac{d\alpha}{du}\right \Vert $ then the arclength of $\alpha
$\ is
\[
S(u)=%
{\displaystyle \int \limits_{0}^{u}}
\left \Vert \frac{\partial \alpha}{\partial u}\right \Vert du=%
{\displaystyle \int \limits_{0}^{u}}
vdu.
\]
The operator $\frac{\partial}{\partial s}$\ is given with respect to\ $u$\ by
\begin{equation}
\frac{\partial}{\partial s}=\frac{1}{v}\frac{\partial}{\partial u}.\label{3.1}%
\end{equation}
Thus, the arclength is $ds=vdu$.

\begin{definition}
Any flow of the curve can be expressed following form:%
\[
\frac{\partial \alpha}{\partial t}=%
{\displaystyle \sum \limits_{i=1}^{n}}
f_{i}V_{i}%
\]
where $f_{i}$ denotes the $i^{th}$ scalar speed of the curve Let the arclength
variation be%
\[
S(u,t)=%
{\displaystyle \int \limits_{0}^{u}}
vdu.
\]
In the Euclidean space the requirement that the curve not be subject to any
elongation or compression can be expressed by the condition
\begin{equation}
\frac{\partial}{\partial t}S(u,t)=%
{\displaystyle \int \limits_{0}^{u}}
\frac{\partial v}{\partial t}du=0,\text{ \ }u\in \left[  0,l\right]
.\label{3.2}%
\end{equation}

\end{definition}

\begin{definition}
A curve evolution $\alpha(u,t)$\ and its flow $\frac{\partial \alpha}{\partial
t}$ are said to be inextensible if%
\[
\frac{\partial}{\partial t}\left \Vert \frac{\partial \alpha}{\partial
u}\right \Vert =0.
\]
Now, we research the necessary and sufficient condition for inelastic curve
flow. For this reason, we need to the following Lemma.

\begin{lemma}
Let $\frac{\partial \alpha}{\partial t}=%
{\displaystyle \sum \limits_{i=1}^{n}}
f_{i}V_{i}$ be a smooth flow of the curve $\alpha$. The flow is inextensible
if and only if
\begin{equation}
\frac{\partial v}{\partial t}=\frac{\partial f_{1}}{\partial u}-f_{2}%
vk_{1}.\label{3.3}%
\end{equation}

\begin{proof}
Since $\frac{\partial}{\partial u}$\ and $\frac{\partial}{\partial t}%
$\ commute and $v^{2}=\left \langle \frac{\partial \alpha}{\partial u}%
,\frac{\partial \alpha}{\partial u}\right \rangle ,$\ we have%
\begin{align*}
2v\frac{\partial v}{\partial t}  & =\frac{\partial}{\partial t}\left \langle
\frac{\partial \alpha}{\partial u},\frac{\partial \alpha}{\partial
u}\right \rangle \\
& =2\left \langle \frac{\partial \alpha}{\partial u},\frac{\partial}{\partial
u}\left(
{\displaystyle \sum \limits_{i=1}^{n}}
f_{i}V_{i}\right)  \right \rangle \\
& =2\left \langle vV_{1},%
{\displaystyle \sum \limits_{i=1}^{n}}
\frac{\partial f_{i}}{\partial u}V_{i}+%
{\displaystyle \sum \limits_{i=1}^{n}}
f_{i}\frac{\partial V_{i}}{\partial u}\right \rangle \\
& =2\left \langle vV_{1},\frac{\partial f_{1}}{\partial u}V_{1}+f_{1}%
\frac{\partial V_{1}}{\partial u}+...+\frac{\partial f_{n}}{\partial u}%
V_{n}+f_{n}\frac{\partial V_{n}}{\partial u}\right \rangle \\
& =2\left \langle vV_{1},\frac{\partial f_{1}}{\partial u}V_{1}+f_{1}%
vk_{1}V_{2}+...+\frac{\partial f_{n}}{\partial u}V_{n}-f_{n}vk_{n-1}%
V_{n-1}\right \rangle \\
& =2\left(  \frac{\partial f_{1}}{\partial u}-f_{2}vk_{1}\right)  .
\end{align*}
Thus, we reach%
\[
\frac{\partial v}{\partial t}=\frac{\partial f_{1}}{\partial u}-f_{2}vk_{1}.
\]

\end{proof}
\end{lemma}
\end{definition}

\begin{theorem}
\label{teo3.1}\bigskip Let $\left \{  V_{1},V_{2},...,V_{n}\right \}  $ be the
moving Frenet frame of the curve $\alpha$ and $\frac{\partial \alpha}{\partial
t}=%
{\displaystyle \sum \limits_{i=1}^{n}}
f_{i}V_{i}$\ be a differentiable flow of $\alpha$ in $E^{n\text{ }}$.Then the
flow is inextensible if and only if
\begin{equation}
\frac{\partial f_{1}}{\partial s}=f_{2}k_{1}.\label{3.4}%
\end{equation}

\end{theorem}

\begin{proof}
Suppose that the curve flow is inextensible. From equations (\ref{3.2}) and
(\ref{3.3}) for $u\in \left[  0,l\right]  $, we see that
\[
\frac{\partial}{\partial t}S(u,t)=%
{\displaystyle \int \limits_{0}^{u}}
\frac{\partial v}{\partial t}du=%
{\displaystyle \int \limits_{0}^{u}}
\left(  \frac{\partial f_{1}}{\partial u}-f_{2}vk_{1}\right)  du=0.
\]
Thus, it can be see that%
\[
\frac{\partial f_{1}}{\partial u}-f_{2}vk_{1}=0.
\]
Considering the last equation and (\ref{3.1}), we reach
\[
\frac{\partial f_{1}}{\partial s}=f_{2}k_{1}.
\]
Conversely, following similar way as above, the proof is completed.

Now, we restrict ourselves to arclength parameterized curves. That is,
$v=1$\ and the local coordinate $u$\ corresponds to the curve arclength $s$.
We require the following Lemma

\begin{lemma}
Let $\left \{  V_{1},V_{2},...,V_{n}\right \}  $\ be the moving Frenet frame of
the curve $\alpha$. Then, the differentions of $\left \{  V_{1},V_{2}%
,...,V_{n}\right \}  $\ with respect to $t$\ is%
\begin{align*}
\frac{\partial V_{1}}{\partial t}  & =\left[
{\displaystyle \sum \limits_{i=2}^{n-1}}
\left(  f_{i-1}k_{i-1}+\frac{\partial f_{i}}{\partial s}-f_{i+1}k_{i}\right)
V_{i}\right]  +\left(  f_{n-1}k_{n-1}+\frac{\partial f_{n}}{\partial
s}\right)  V_{n},\\
\frac{\partial V_{j}}{\partial t}  & =-\left(  f_{j-1}k_{j-1}+\frac{\partial
f_{j}}{\partial s}-f_{j+1}k_{j}\right)  V_{1}+\left[
{\displaystyle \sum \limits_{\substack{k=2 \\k\neq i}}^{n}}
\Psi_{kj}V_{k}\right]  ,\text{ \ }1<j<n,\\
\frac{\partial V_{n}}{\partial t}  & =-\left(  f_{n-1}k_{n-1}+\frac{\partial
f_{n}}{\partial s}\right)  V_{1}+\left[
{\displaystyle \sum \limits_{k=2}^{n-1}}
\Psi_{kn}V_{k}\right]  ,
\end{align*}
where $\Psi_{kj}=\left \langle \frac{\partial V_{j}}{\partial t},V_{k}%
\right \rangle $\ and $\Psi_{kn}=\left \langle \frac{\partial V_{n}}{\partial
t},V_{k}\right \rangle $.

\begin{proof}
For $\frac{\partial}{\partial t}$\ and $\frac{\partial}{\partial s}$\ commute,
it seen that
\begin{align*}
\frac{\partial V_{1}}{\partial t}  & =\frac{\partial}{\partial t}\left(
\frac{\partial \alpha}{\partial s}\right)  =\frac{\partial}{\partial s}\left(
\frac{\partial \alpha}{\partial t}\right)  =\frac{\partial}{\partial s}\left(
{\displaystyle \sum \limits_{i=1}^{n}}
f_{i}V_{i}\right)  =%
{\displaystyle \sum \limits_{i=1}^{n}}
\frac{\partial f_{i}}{\partial s}V_{i}+%
{\displaystyle \sum \limits_{i=1}^{n}}
f_{i}\frac{\partial V_{i}}{\partial s}\\
& =\frac{\partial f_{1}}{\partial s}V_{1}+f_{1}\frac{\partial V_{1}}{\partial
s}+\frac{\partial f_{2}}{\partial s}V_{2}+f_{2}\frac{\partial V_{2}}{\partial
s}+...+\frac{\partial f_{n}}{\partial s}V_{n}+f_{n}\frac{\partial V_{n}%
}{\partial s}\\
& =\frac{\partial f_{1}}{\partial s}V_{1}+f_{1}k_{1}V_{2}+\frac{\partial
f_{2}}{\partial s}V_{2}+f_{2}\left(  -k_{1}V_{1}+k_{2}V_{3}\right)
+...+\frac{\partial f_{n}}{\partial s}V_{n}-f_{n}k_{n-1}V_{n-1}.
\end{align*}
Substituting the equation (\ref{3.4}) into the last equation and using Theorem
\ref{teo3.1}., we have%
\[
\frac{\partial V_{1}}{\partial t}=\left[
{\displaystyle \sum \limits_{i=2}^{n-1}}
\left(  f_{i-1}k_{i-1}+\frac{\partial f_{i}}{\partial s}-f_{i+1}k_{i}\right)
V_{i}\right]  +\left(  f_{n-1}k_{n-1}+\frac{\partial f_{n}}{\partial
s}\right)  V_{n}.
\]
Now, let us differentiate the Frenet frame with respect to $t$ for $1<j<n$\ as
follows;
\begin{align}
0  & =\frac{\partial}{\partial t}\left \langle V_{1},V_{j}\right \rangle
=\left \langle \frac{\partial V_{1}}{\partial t},V_{j}\right \rangle
+\left \langle V_{1},\frac{\partial V_{j}}{\partial t}\right \rangle \nonumber \\
& =\left(  f_{i-1}k_{i-1}+\frac{\partial f_{i}}{\partial s}-f_{i+1}%
k_{i}\right)  +\left \langle V_{1},\frac{\partial V_{j}}{\partial
t}\right \rangle .\label{3.5}%
\end{align}
From (\ref{3.5}), we have obtain
\[
\frac{\partial V_{j}}{\partial t}=-\left(  f_{j-1}k_{j-1}+\frac{\partial
f_{j}}{\partial s}-f_{j+1}k_{j}\right)  V_{1}+\left[
{\displaystyle \sum \limits_{\substack{k=2 \\k\neq j}}^{n}}
\Psi_{kj}V_{k}\right]  .
\]
Lastly, considering $\left \langle V_{1},V_{n}\right \rangle =0$\ and following
similar way as above, we reach%
\[
\frac{\partial V_{n}}{\partial t}=-\left(  f_{n-1}k_{n-1}+\frac{\partial
f_{n}}{\partial s}\right)  V_{1}+\left[
{\displaystyle \sum \limits_{k=2}^{n-1}}
\Psi_{kn}V_{k}\right]  .
\]

\begin{theorem}
Suppose that the curve flow $\frac{\partial \alpha}{\partial t}=%
{\displaystyle \sum \limits_{i=1}^{n}}
f_{i}V_{i}$\ is inextensible. Then the following system of partial
differential equations holds:%
\begin{align*}
\frac{\partial k_{1}}{\partial t}  & =f_{2}k_{1}^{2}+f_{1}\frac{\partial
k_{1}}{\partial s}+\frac{\partial^{2}f_{2}}{\partial s^{2}}-2\frac{\partial
f_{2}}{\partial s}k_{2}-f_{3}\frac{\partial k_{2}}{\partial s}-f_{2}k_{2}%
^{2}-f_{4}k_{3}k_{2}\\
\frac{\partial k_{i-1}}{\partial t}  & =-\frac{\partial \Psi_{(i-1)i}}{\partial
s}-\Psi_{(i-2)i}k_{i-2}\\
\frac{\partial k_{i}}{\partial t}  & =\frac{\partial \Psi_{(i-1)i}}{\partial
s}-\Psi_{(i+2)i}k_{i+2}\\
\frac{\partial k_{n-1}}{\partial t}  & =-\frac{\partial \Psi_{(n-1)n}}{\partial
s}-\Psi_{(n-2)n}k_{n-2}.
\end{align*}

\end{theorem}
\end{proof}
\end{lemma}
\end{proof}

\begin{proof}
\bigskip Since $\frac{\partial}{\partial s}\frac{\partial V_{1}}{\partial
t}=\frac{\partial}{\partial t}\frac{\partial V_{1}}{\partial s}$, we get%
\begin{align*}
\frac{\partial}{\partial s}\frac{\partial V_{1}}{\partial t}  & =\frac
{\partial}{\partial s}\left[
{\displaystyle \sum \limits_{i=2}^{n-1}}
\left(  f_{i-1}k_{i-1}+\frac{\partial f_{i}}{\partial s}-f_{i+1}k_{i}\right)
V_{i}+\left(  f_{n-1}k_{n-1}+\frac{\partial f_{n}}{\partial s}\right)
V_{n}\right] \\
& =%
{\displaystyle \sum \limits_{i=2}^{n-1}}
\left[  \left(  \frac{\partial f_{i-1}}{\partial s}k_{i-1}+f_{i-1}%
\frac{\partial k_{i-1}}{\partial s}+\frac{\partial^{2}f_{i}}{\partial s^{2}%
}-\frac{\partial f_{i+1}}{\partial s}k_{i}-f_{i+1}\frac{\partial k_{i}%
}{\partial s}\right)  V_{i}\right] \\
& +%
{\displaystyle \sum \limits_{i=2}^{n-1}}
\left[  \left(  f_{i-1}k_{i-1}+\frac{\partial f_{i}}{\partial s}-f_{i+1}%
k_{i}\right)  \frac{\partial V_{i}}{\partial s}\right] \\
& +\left(  \frac{\partial f_{n-1}}{\partial s}k_{n-1}+f_{n-1}\frac{\partial
k_{n-1}}{\partial s}+\frac{\partial^{2}f_{n}}{\partial s^{2}}\right)
V_{n}+\left(  f_{n-1}k_{n-1}+\frac{\partial f_{n}}{\partial s}\right)
\frac{\partial V_{n}}{\partial s}%
\end{align*}
while
\[
\frac{\partial}{\partial t}\frac{\partial V_{1}}{\partial s}=\frac{\partial
}{\partial t}\left(  k_{1}V_{2}\right)  =\frac{\partial k_{1}}{\partial
t}V_{2}+k_{1}\frac{\partial V_{2}}{\partial t}.
\]
Thus, from the both of above two equations, we reach%
\[
\frac{\partial k_{1}}{\partial t}=f_{2}k_{1}^{2}+f_{1}\frac{\partial k_{1}%
}{\partial s}+\frac{\partial^{2}f_{2}}{\partial s^{2}}-2\frac{\partial f_{3}%
}{\partial s}k_{2}-f_{3}\frac{\partial k_{2}}{\partial s}-f_{2}k_{2}^{2}%
-f_{4}k_{3}k_{2}.
\]
For $1<i<n$, noting that $\frac{\partial}{\partial s}\frac{\partial V_{i}%
}{\partial t}=\frac{\partial}{\partial t}\frac{\partial V_{i}}{\partial s}$,
it is seen that%
\begin{align*}
\frac{\partial}{\partial s}\frac{\partial V_{i}}{\partial t}  & =\frac
{\partial}{\partial s}\left[  -\left(  f_{i-1}k_{i-1}+\frac{\partial f_{i}%
}{\partial s}-f_{i+1}k_{i}\right)  V_{1}+%
{\displaystyle \sum \limits_{k=2}^{n}}
\Psi_{kj}V_{k}\right] \\
& =-\left(  \frac{\partial f_{i-1}}{\partial s}k_{i-1}+f_{i-1}\frac{\partial
k_{i-1}}{\partial s}+\frac{\partial^{2}f_{i}}{\partial s^{2}}-\frac{\partial
f_{i+1}}{\partial s}k_{i}-f_{i+1}\frac{\partial k_{i}}{\partial s}\right)
V_{1}\\
& +\left(  f_{i-1}k_{i-1}+\frac{\partial f_{i}}{\partial s}-f_{i+1}%
k_{i}\right)  \frac{\partial V_{1}}{\partial s}+%
{\displaystyle \sum \limits_{\substack{k=2 \\k\neq i}}^{n}}
\left(  \frac{\partial \Psi_{ki}}{\partial s}V_{k}+\Psi_{ki}\frac{\partial
V_{k}}{\partial s}\right)
\end{align*}
while%
\[
\frac{\partial}{\partial t}\frac{\partial V_{i}}{\partial s}=\frac{\partial
}{\partial t}\left(  -k_{i-1}V_{i-1}+k_{i}V_{i+1}\right)  =-\frac{\partial
k_{i-1}}{\partial t}V_{i-1}-k_{i-1}\frac{\partial V_{i-1}}{\partial t}%
+\frac{\partial k_{i}}{\partial t}V_{i+1}+k_{i}\frac{\partial V_{i+1}%
}{\partial t}.
\]
Thus, we obtain
\[
\frac{\partial k_{i-1}}{\partial t}=-\frac{\partial \Psi_{(i-1)i}}{\partial
s}-\Psi_{(i-2)i}k_{i-2}%
\]
and%
\[
\frac{\partial k_{i}}{\partial t}=\frac{\partial \Psi_{(i+1)i}}{\partial
s}-\Psi_{(i+2)i}k_{i+1}.
\]
Lastly, considering $\frac{\partial}{\partial s}\frac{\partial V_{n}}{\partial
t}=\frac{\partial}{\partial t}\frac{\partial V_{n}}{\partial s}$\ and
following similar way as above, we reach\bigskip%
\[
\frac{\partial k_{n-1}}{\partial t}=-\frac{\partial \Psi_{(n-1)n}}{\partial
s}-\Psi_{(n-2)n}k_{n-2}.
\]

\end{proof}

\end{document}